\documentclass{amsart}
\usepackage{a4wide}
\usepackage{amsfonts}
\usepackage{url}
\usepackage{listings}
\theoremstyle{plain}
\usepackage[osf,sc]{mathpazo}
\newtheorem{theorem}{Theorem}
\usepackage{tcolorbox}

\newtheorem{lemma}{Lemma}

\theoremstyle{definition}
\newtheorem{definition}{Definition}

\allowdisplaybreaks

\newcommand{\bell}{\textup{B}}



\author{Sumit Kumar Jha}
\address{International Institute of Information Technology\\
Hyderabad-500032, India}
\email{kumarjha.sumit@research.iiit.ac.in} 

\begin{document}
	
\title[Triangular number decomposition]{An identity involving number of representations of $n$ as a sum of $r$ triangular numbers}

\begin{abstract}
Let $\sum_{d|n}$ denote sum over divisors of a positive integer $n$, and $t_{r}(n)$ denote the number of representations of $n$ as a sum of $r$ triangular numbers. Then we prove that
$$
\sum_{d|n}\frac{1+2\,(-1)^{d}}{d}=\sum_{r=1}^{n}\frac{(-1)^{r}}{r}\, \binom{n}{r}\, t_{r}(n)
$$
using a result of Ono, Robbins and Wahl.
\end{abstract}

\maketitle

\section{Main result}
In the following, $\sum_{d|n}$ denotes the sum over divisors of a positive integer $n$.
\begin{definition}
Let $\Psi(q)$ be the following infinite series 
$$\Psi(q):=\sum_{n=0}^{\infty}q^{n(n+1)/2}=1+q+q^{3}+q^{6}+\cdots$$ where $|q|<1$.
\end{definition}
\begin{definition}
For any positive integer $n$, the numbers $n(n+1)/2$ are the \emph{triangular numbers}. Suppose $n$ and $r$ are positive integers. We let $t_{r}(n)$ denote the number of representations of $n$ as a sum of $r$ triangular numbers where representations with different orders are counted as unique. Let $t_{r}(0)=1$.
\end{definition}
For example, $t_{2}(7)=2$ since $7=1+6=6+1$.\par 
We note that
$$
\Psi^{r}(q)=\sum_{n=0}^{\infty}t_{r}(n)\, q^{n}.
$$
Our aim is to derive the following identity.
\begin{theorem}
\label{main}
For all positive integers $n$ we have
\begin{equation}
\label{maineq}
\sum_{d|n}\frac{1+2\,(-1)^{d}}{d}=\sum_{r=1}^{n}\frac{(-1)^{r}}{r}\, \binom{n}{r}\, t_{r}(n).
\end{equation}
\end{theorem}
We use the following result due to Ono, Robins and Wahl \cite[Proposition 1]{Wahl}:
\begin{equation}
    \label{ono}
    \Psi(q)=\prod_{j=1}^{\infty}\frac{(1-q^{2j})^2}{(1-q^{j})}=\prod_{j=1}^{\infty}(1+q^{j})^2\, (1-q^j)
\end{equation}
where $|q|<1$.\par 
We further require following two lemmas for our proof.
\begin{lemma}
For all positive integers $n$ we have
\begin{equation}
\label{lem1}
\sum_{d|n}\frac{1+2\,(-1)^{d}}{d}=\frac{1}{n!}\sum_{k=1}^n (-1)^{k}\, (k-1)!\, B_{n,k}\left(\Psi'(0),\Psi''(0),\dots,\Psi^{(n-k+1)}(0)\right)
\end{equation}
where $B_{n,k}\equiv\bell_{n,k}(x_1,x_2,\dotsc,x_{n-k+1})$ are the partial Bell polynomials defined by \cite[p. 134]{Comtet}
\begin{equation*}
\bell_{n,k}(x_1,x_2,\dotsc,x_{n-k+1})=\sum_{\substack{1\le i\le n,\ell_i\in\mathbb{N}\\ \sum_{i=1}^ni\ell_i=n\\ \sum_{i=1}^n\ell_i=k}}\frac{n!}{\prod_{i=1}^{n-k+1}\ell_i!} \prod_{i=1}^{n-k+1}\Bigl(\frac{x_i}{i!}\Bigr)^{\ell_i}.
\end{equation*}
\end{lemma}
\begin{proof}
Taking logarithm on both sides of Equation \eqref{ono} we easily see that
\begin{align*}
\log(\Psi(q))&=\sum_{j=1}^{\infty}2\,\log(1+q^{j})+\sum_{j=1}^{\infty}\log(1-q^{j})\\
&=-\sum_{j=1}^{\infty}\sum_{l=1}^{\infty}2\,\frac{(-1)^{l}\,q^{lj}}{l}-\sum_{j'=1}^{\infty}\sum_{l'=1}^{\infty}\frac{q^{l'j'}}{l'}\\
&=-\sum_{n=1}^{\infty}q^{n}\,\sum_{d|n} \frac{1+2\,(-1)^{d}}{d}.
\end{align*}
Let $f(q)=\log{q}$. Using Fa\`{a} di Bruno's formula \cite[p. 137]{Comtet} we have
\begin{equation}
\label{faa}
{d^n \over dq^n} f(\Psi(q)) = \sum_{k=1}^n f^{(k)}(\Psi(q))\cdot B_{n,k}\left(\Psi'(q),\Psi''(q),\dots,\Psi^{(n-k+1)}(q)\right).
\end{equation}
Since $f^{(k)}(q)=\frac{(-1)^{k-1}\,(k-1)!}{q^{k}}$ and $\Psi(0)=1$, letting $q\rightarrow 0$ in the above equation gives us Equation \eqref{lem1}.
\end{proof}
\begin{lemma}
We have, for positive integers $n,k$,
\begin{equation}
\label{lem2}
B_{n,k}\left(\Psi'(0),\Psi''(0),\dots,\Psi^{(n-k+1)}(0)\right)=\frac{n!}{k!}\sum_{r=1}^{k}(-1)^{k-r}\binom{k}{r}t_{r}(n)
\end{equation}
\end{lemma}
\begin{proof}
We start with the generating function for the partial Bell polynomials \cite[Equation (3a') on p. 133]{Comtet}, and proceed as follows
\begin{align*}
{\displaystyle \sum _{n=k}^{\infty }B_{n,k}\left(\Psi'(0),\Psi''(0),\dots,\Psi^{(n-k+1)}(0)\right){\frac {q^{n}}{n!}}}
&= {\frac {1}{k!}}\left(\sum _{j=1}^{\infty }\Psi^{(j)}(0){\frac {q^{j}}{j!}}\right)^{k} \\
&=\frac{1}{k!}(\Psi(q)-1)^{k}\\
&=\frac{1}{k!}\sum_{r=0}^{k}(-1)^{k-r}\binom{k}{r}\Psi(q)^{r}\\
&=\frac{1}{k!}\sum_{r=0}^{k}(-1)^{k-r}\binom{k}{r}\sum_{n=0}^{\infty}t_{r}(n)q^{n}
\end{align*}
to conclude Equation \eqref{lem2}.
\end{proof}
\begin{proof}[Proof of Theorem \ref{main}]
We combine Equation \eqref{lem1} and Equation \eqref{lem2} to obtain
\begin{align*}
\sum_{d|n} \frac{1+2\,(-1)^{d}}{d}&=\sum_{k=1}^{n}\frac{1}{k}\sum_{r=1}^{k}(-1)^{r}\, \binom{k}{r}\,t_{r}(n)\\
&=\sum_{r=1}^{n}(-1)^{r}\, t_{r}(n)\sum_{k=r}^{n}\frac{1}{k}\binom{k}{r}\\
&=\sum_{r=1}^{n}\frac{(-1)^{r}}{r}\,\binom{n}{r}\,t_{r}(n),
\end{align*}
where we used the identity $\sum_{k=r}^{n}\frac{1}{k}\binom{k}{r}=\frac{1}{r}\binom{n}{r}$ which can be proved using the Pascal's formula
$$
\binom{k}{r-1}=\binom{k+1}{r}-\binom{k}{r}.
$$
This concludes proof of our main result Equation \eqref{maineq}.
\end{proof}

\end{document}